\documentclass[11pt,reqno]{amsart}
\usepackage{amsmath,amssymb,graphicx,mathrsfs,amsopn,stmaryrd,amsthm,color,bm}
\usepackage[colorlinks=true,citecolor=blue,linkcolor=blue,pagebackref]{hyperref}
\usepackage[final]{showlabels}
\usepackage[english]{babel}
\usepackage[mathscr]{euscript}
\usepackage{xcolor}
\usepackage{geometry}
\usepackage{dsfont}
\usepackage{tikz}    
\usetikzlibrary{arrows,matrix,decorations.markings,shapes.geometric}   
\usetikzlibrary{decorations.pathreplacing}

\usepackage{xy}
\allowdisplaybreaks
\xyoption{all}
\numberwithin{equation}{section}
\newtheorem{thm}{Theorem}[section]
\newtheorem{prop}[thm]{Proposition}
\newtheorem{lem}[thm]{Lemma}
\newtheorem{cor}[thm]{Corollary}

\theoremstyle{definition} 
\newtheorem{eg}[thm]{Example}
\newtheorem{dfn}[thm]{Definition}
\theoremstyle{remark}
\newtheorem{rem}[thm]{Remark}

\newcommand{\beq}{\begin{equation}}
\newcommand{\eeq}{\end{equation}}
\newcommand{\be}{\begin{equation*}}
\newcommand{\ee}{\end{equation*}}

\newcommand{\C}{\mathbb{C}}

\newcommand{\Z}{\mathbb{Z}}

\newcommand{\mc}{\mathcal}

\newcommand{\sfv}{\mathsf{v}}
\newcommand{\sfh}{\mathsf{h}}
\newcommand{\sfe}{\mathsf{e}}
\newcommand{\sff}{\mathsf{f}}
\newcommand{\sfk}{\mathsf{k}}
\newcommand{\sfE}{\mathsf{E}}
\newcommand{\sfF}{\mathsf{F}}
\newcommand{\sfK}{\mathsf{K}}

\newcommand{\gl}{\mathfrak{gl}}

\newcommand{\fksl}{\mathfrak{sl}}
\newcommand{\fkS}{\mathfrak{S}}

\newcommand{\End}{\mathrm{End}}

\newcommand{\scrF}{\mathscr{F}}

\newcommand{\tl}{\tilde}
\newcommand{\gge}{\geqslant}
\newcommand{\lle}{\leqslant}
\newcommand{\la}{\lambda}

\newcommand{\bla}{\bm\lambda}

\newcommand{\wt}{\widehat}    
\newcommand{\ka}{\kappa}
\newcommand{\ve}{\varepsilon}

\newcommand{\qedd}{\tag*{$\square$}}

\newcommand{\daha}{{\ddot{\mathbb H}_{\ell}}}

\newcommand{\aha}{{\dot{\mathbb H}_{\ell}}}
\newcommand{\ha}{{{\mathbb H}_{\ell}}}
\newcommand{\lb}[1]{\llbracket #1 \rrbracket}
\newcommand{\s}{{\bm s}}
\newcommand{\uqhsls}{{\mathscr U_q(\wt{\fksl}_{\bm s})}}
\newcommand{\uqsls}{{\mathscr U_q({\fksl}_{\bm s})}}
\newcommand{\Sym}{{\mathrm{Sym}}}
\newcommand{\Es}{{\mathscr{E}_{\bm s}}}
\newcommand{\Vs}{{\mathscr{V}_{\bm s}}}
\newcommand{\Vsl}{{\mathscr V_\s^{\otimes \ell}}}
\newcommand{\uvs}{{\mathscr U_q^\sfv(\widehat{\fksl}_\s)}}
\newcommand{\uhs}{{\mathscr U_q^\sfh(\widehat{\fksl}_\s)}}

\begin{document}
\pagestyle{myheadings}
\setcounter{page}{1}

\title[Schur-Weyl duality for quantum toroidal superalgebras]{Schur-Weyl duality for quantum toroidal superalgebras}

\author{Kang Lu}
\address{K.L.: Department of Mathematics, University of Denver
\newline
\strut\kern\parindent \quad \ \ \   \ \   Department of Mathematics, University of Virginia}\email{\sf kang.lu@virginia.edu}

\begin{abstract} 
We establish the Schur-Weyl type duality between double affine Hecke algebras and quantum toroidal superalgebras, generalizing the well known result of Vasserot-Varagnolo \cite{VV96} to the super case.
		
		\medskip 
		
		\noindent
		{\bf Keywords:} Schur-Weyl duality, double affine Hecke algebra, quantum toroidal superalgebra 
		
\end{abstract}

\maketitle

%\tableofcontents

\thispagestyle{empty}
\section{Introduction}	
In the last 30 years, quantum toroidal algebras \cite{GKV95} and double affine Hecke algebras (DAHA for short) \cite{Che92} are central objects in the area of representation theory. They have rich representation theory and also many important applications in algebra, combinatorics, geometry, and mathematical physics. In \cite{VV96}, it is shown that these two remarkable algebras are related via Schur-Weyl duality. 

Recently, quantum toroidal superalgebras associated to $\fksl_{m|n}$ for arbitrary root systems were introduced in \cite{BM19}. A related geometric construction of the Drinfeld half of quantum toroidal superalgebras using the deformed K-theoretic Hall algebra of a quiver with potential is given in \cite{VV22}. The present paper is devoted to establishing the Schur-Weyl duality between double affine Hecke algebras and quantum toroidal superalgebras, generalizing the well known result of Vasserot-Varagnolo \cite{VV96} to the super case. We expect that this duality could be an important tool to study representations of quantum toroidal superalgebras, cf. e.g. \cite{BL22}, and to obtain  results for super case from the (certain) known results in the even case, see e.g. \cite[Section 4]{LM20}.

\medskip 

Schur-Weyl duality, being one of the most important and beautiful classical results in representation theory, is the equivalence between the category of modules over the symmetric group $\fkS_\ell$ and the category of modules of level $\ell$ over the Lie algebra $\fksl_{n}$ for $\ell<n$. Since the introduction of quantum groups in the 1980s, it is interesting and important to generalize Schur-Weyl duality in the quantum setting. In fact, similar equivalences or related results have been established between finite Hecke algebras and quantum enveloping (super)algebras \cite{Jim86,Moo03,Mit06}, between degenerate affine Hecke algebras and (super) Yangians \cite{Dri86,Ara,LM20,Lu21}, between affine Hecke algebras and quantum affine (super)algebras \cite{Che87,GRV94,CP96,Fli18,KL20}, between double affine Hecke algebras (also called elliptic Cherednik algebras) and quantum toroidal algebras \cite{VV96},
between trigonometric Cherednik algebras and affine Yangians \cite{Gua05,Gua07},
and between rational Cherednik algebras and deformed double current algebras \cite{Gua05,Gua07}. These relations can be summarized  by combining the table below and the table in \cite[Introduction]{Rou05}\footnote{This is borrowed from N. Guay's talk in Representations and Lie Theory Seminar at Ohio State University.}.

\vskip -4cm

$$
\scalebox{0.94}{
\xy
(0,-40) *{\text{
\begin{tabular}{|l|l|l|}
\hline
quantum (super)algebras & quantum affine (super)algebras & quantum toroidal (super)algebras \\ \hline
 & (super) Yangians & affine (super) Yangians \\ \hline
 & & deformed double current (super)algebras \\ \hline
\end{tabular}
}}

\endxy
}
$$
More specifically, the (super)algebras in this table are the dual (superalgebras) for the corresponding algebras in the table of \cite[Introduction]{Rou05}.

It is also interesting to generalize the last two cases to the super setting. Note that the affine super Yangians (of type A associated to the standard root system) have been introduced in \cite{Ued19} while deformed double current superalgebras are not discussed in the literature yet.

\medskip

We almost follow the arguments used in \cite{VV96} except \cite[Theorem 3.3]{VV96} which was deduced by using the braid group action on algebras and integrable modules. A similar description of \cite[Theorem 3.3]{VV96} using affine Hecke algebra does not seem to work in the super case. In order to generalize \cite[Theorem 3.3]{VV96} to the super case, a modification of the action of affine Hecke algebra on polynomial tensor representation is probably needed, see \cite[Section 4]{GRV94}. We obtain similar results by investigating the coproduct of the quantum affine superalgebra, see Propositions \ref{prop:copro}, \ref{prop:tensor-vect-rep}.

\medskip

{\bf Acknowledgments.} The author thanks E. Mukhin for stimulating discussions and the referee for a careful reading of the manuscript and for pointing out a gap in the original proof of Proposition \ref{prop:tensor-vect-rep}. %This work was partially supported by grants from the Simons Foundation \#353831 and \#709444.

\section{Preliminaries}
We fix $m,n\in\Z_{\gge 0}$ such that $m\ne n$ and set $\kappa=m+n$. Set $I=\{1,2,\dots,\kappa-1\}$ and $\hat I=\{0,1,\dots,\kappa-1\}$. Fix $q\in \C^\times$ to be not a root of unity.

\subsection{Double affine Hecke algebras}
Let $\zeta\in \C^\times$ and $\ell\in\Z_{>0}$.
\begin{dfn}[\cite{Che92}]\label{def:DAHA}
	The \emph{double affine Hecke algebra} (or {\it elliptic Cherednik algebra}) of type $\gl_\ell$, denoted by $\ddot{\mathbb H}_{\ell}$, is the unital associative algebra with the generators $T_i^{\pm1}$, $X_j^{\pm 1}$, $Y_j^{\pm1}$, $1\lle i<\ell$, $1\lle j\lle \ell$, and the relations:
	$$
	T_iT_i^{-1}=T_i^{-1}T_i=1,\quad (T_i+1)(T_i- q^2)=0,\quad T_iT_{i+1}T_i=T_{i+1}T_iT_{i+1},
	$$ 
	$$
	X_0Y_1=\zeta Y_1X_0,\quad X_iX_j=X_iX_j,\quad Y_iY_j=Y_jY_i,\quad X_iX_i^{-1}=X_i^{-1}X_i=Y_iY_i^{-1}=Y_i^{-1}Y_i=1,
	$$
	$$
	T_iX_iT_i= q^2 X_{i+1},\quad T_i^{-1}Y_iT_i^{-1}= q^{-2}Y_{i+1},\quad X_2Y_1^{-1}X_2^{-1}Y_1= q^{-2} T_1^2,
	$$
	$$
	T_iT_j=T_jT_i\quad \text{if } |i-j|>1,\qquad X_jT_i=T_iX_j,\quad Y_jT_i=T_iY_j\quad \text{if }j\ne i,i+1,
	$$
where $X_0=X_1X_2\cdots X_\ell$.	\qed
\end{dfn} 

Here we set $\mathbf y=1$ in \cite[Definition 1.1]{VV96}.

\medskip

Let $\fkS_\ell$ be the symmetric group permuting the set $\{1,2,\dots,\ell\}$. Given an element $w\in \fkS_\ell$, let $T_w\in \ddot{\mathbb H}_{\ell}$ be the element defined in terms of a reduced expression of $w$.

For a sequence of $\ell$ integers $\bm r=(r_1,\dots,r_\ell)$, set $X^{\bm r}:=X_1^{r_1}\cdots X_\ell^{r_\ell}$ and $Y^{\bm r}:=Y_1^{r_1}\cdots Y_\ell^{r_\ell}$. It is known that elements $X^{\bm s} Y^{\bm r}T_w$ for all possible $\ell$-tuples $\bm s,\bm r$ and $w\in\fkS_\ell$ form a basis of $\ddot{\mathbb H}_{\ell}$, see e.g. \cite[Theorem 2.6 (a)]{Che92}.

Let $\dot{\mathbb H}_{\ell}^{(1)}$ and $\dot{\mathbb H}_{\ell}^{(2)}$ be the subalgebras of $\ddot{\mathbb H}_{\ell}$ generated by $T_i^{\pm 1}, Y_j^{\pm 1}$ and $T_i^{\pm 1}, X_j^{\pm 1}$, $1\lle i<\ell$, $1\lle j\lle \ell$, respectively. Then $\dot{\mathbb H}_{\ell}^{(1)}$ and $\dot{\mathbb H}_{\ell}^{(2)}$ are isomorphic to the affine Hecke algebra of type $\gl_\ell$, which we denote it by $\aha$. Similarly, the subalgebra generated by $T_i^{\pm 1}$, $1\lle i<\ell$, is isomorphic to the Hecke algebra of type $\gl_\ell$ and we denote it by $\mathbb H_{\ell}$.

For $1\lle i\lle j<\ell$, we use the convenient notation,
\beq\label{eq:Tij}
T_{i,j}:=T_{i}T_{i+1}\cdots T_j,\qquad T_{j,i}:=T_{j}T_{j-1}\cdots T_i.
\eeq

\medskip

The double affine Hecke algebra $\ddot{\mathbb H}_{\ell}$ admits another well-known presentation as follows.

\begin{prop}\label{prop:presentation}
	The double affine Hecke algebra $\ddot{\mathbb H}_{\ell}$ is the unital associative algebra with the generators $Q^{\pm 1}$, $T_i^{\pm1}$, $Y_j^{\pm1}$, $1\lle i<\ell$, $1\lle j\lle \ell$, and the relations:
	$$
	T_iT_i^{-1}=T_i^{-1}T_i=1,\quad (T_i+1)(T_i- q^2)=0,\quad T_iT_{i+1}T_i=T_{i+1}T_iT_{i+1},
	$$ 
	$$
QQ^{-1}=Q^{-1}Q=1,\quad Y_iY_j=Y_jY_i,\quad T_i^{-1}Y_iT_i^{-1}=q^{-2}Y_{i+1},\quad Y_iY_i^{-1}=Y_i^{-1}Y_i=1,
	$$
	$$
	T_iT_j=T_jT_i\quad \text{if } |i-j|>1,\quad Y_jT_i=T_iY_j\quad \text{if }j\ne i,i+1,
	$$
    $$
    QT_{i-1}Q^{-1}=T_i\ (1<i<\ell-1),\quad Q^2T_{\ell-1}Q^{-2}=T_{1},
    $$
    $$
    QY_{i}Q^{-1}= Y_{i+1}\ (1\lle i\lle \ell-1),\quad QY_{\ell}Q^{-1}=\zeta Y_{1}.
    $$
Here $Q$ is identified with $X_1T_{1,\ell-1}$ in Definition \ref{def:DAHA}. \qed
\end{prop}

For $1\lle i\lle j<\ell$ and $r<\ell$, set 
\beq\label{eq:Q-P}
Q_{i,j}:=X_iT_{i,j}\in \daha,\qquad P_r:=Q_{\ell-r,\ell-1}\cdots Q_{2,r+1}Q_{1,r}\in\daha.
\eeq
We shall need the following lemmas later. 
\begin{lem}[\cite{VV96}]\label{lem:Q}
	If $i\lle a\lle j$ and $i<b<j$, then 
	\[
	Q_{i,j}Y_aQ_{i,j}^{-1}= Y_{a+1},\qquad Q_{i,j}T_{b-1}Q_{i,j}^{-1}=T_{b}.\qedd
	\]
\end{lem}

\begin{lem}[\cite{VV96}]\label{lem:P}
	If $r<a+1$ and $r<b<\ell$, then 
	\[
	P_rY_{a+1}P_r^{-1}=\zeta Y_{a-r+1},\qquad P_rT_{b}P_r^{-1}=T_{b-r}.\qedd
	\]
\end{lem}

\subsection{Quantum affine superalgebras}
Let $\bm s=(s_1,\dots,s_{\ka})$ where $s_i\in \{\pm 1\}$ and the occurrence of $1$ is exactly $m$. We call such a sequence $\bm s$ a \emph{parity sequence}. Denote the set of all parity sequences by $S_{m|n}$. We call the parity $\bm s=(1,\dots,1,-1,\dots,-1)$ the {\it standard parity sequence}. For an $\bm s\in S_{m|n}$, we extend it to $\bm s=(s_i)_{i\in\Z}$ by enforcing periodicity, $s_{i+\ka}=s_i$ for all $i\in\Z$.

Given a parity sequence $\bm s\in S_{m|n}$, we have the Cartan matrix $A^{\bm s}=(a_{i,j}^{\bm s})_{i,j\in I}$ and the affine Cartan matrix $\hat A^{\bm s}=(a_{i,j}^{\bm s})_{i,j\in \hat I}$ given by
\beq\label{aff cartan}
a_{i,j}^{\bm s}=(s_i+s_{i+1})\delta_{i,j}-s_i\delta_{i,j+1}-s_j\delta_{i+1,j},\qquad i,j\in \hat I.
\eeq

Denote $\mathfrak{sl}_\s$ and $\widehat{\mathfrak{sl}}_\s$ be the Lie superalgebras corresponding to Cartan matrices $A^{\s}$ and $\hat{A}^\s$, respectively. Note that the Lie superalgebras $\mathfrak{sl}_\s$ (resp. $\widehat{\mathfrak{sl}}_\s$) are all isomorphic for all $\s\in S_{m|n}$.

Let $\mathcal P_{\s}$ be the integral lattice spanned by the basis $\ve_i$, $1\lle i\lle \ka$, with a bilinear form on it defined by $\langle \ve_i|\ve_j\rangle =s_i\delta_{i,j}$ for $1\lle i,j\lle \ka$. Set $\alpha_i:=\ve_i-\ve_{i+1}$ for $i\in I$ and let $\mathcal Q_{\s}:=\bigoplus_{i\in I}\Z \alpha_i$ be the root lattice of $\mathfrak{sl}_\s$.

Let $\delta$ be the null root of $\widehat{\mathfrak{sl}}_\s$  such that $\langle \delta|\delta\rangle=\langle \delta|\alpha_i\rangle=0$ for $i\in I$. Let $\alpha_0:=\delta+\ve_\ka-\ve_1$. Then $\langle \alpha_i|\alpha_j\rangle =a_{i,j}^\s$ for $i,j\in\hat I$.

For two homogeneous elements $X,Y$ and $a\in\C$, set $[X,Y]_a=XY-(-1)^{|X||Y|}aYX$. We simply write $[X,Y]$ for $[X,Y]_1$.

\begin{dfn}[\cite{Yam99} Drinfeld-Jimbo presentation]
The quantum affine superalgebra $\uqhsls$ is generated by the {\it Chevalley generators} $e_i$, $f_i$, $t_i^{\pm 1}$, $i\in \hat I$, whose parities are given by $|e_i|=|f_i|=|i|:=(1-s_is_{i+1})/2$, $|t_i^{\pm 1}|=0$, with the defining relations given by
\begin{align*}
	&t_it_j=t_jt_i,\quad t_it_i^{-1}=t_{i}^{-1}t_i=1,\quad t_ie_jt_i^{-1}=q^{a^{\s}_{i,j}}e_j,\quad t_if_jt_i^{-1}=q^{-a^{\s}_{i,j}}f_j,\\
	&[e_i,f_j]=\delta_{i,j}\frac{t_i-t_i^{-1}}{q-q^{-1}},\\
	  & [e_i,e_j]=[f_i,f_j]=0                                                                                           & (a^{\s}_{i,j}=0),           \\
	  & \lb{e_i,\lb{e_i,e_{i\pm 1}}}=\lb{f_i,\lb{f_i,f_{i\pm 1}}}=0                                                     & (a^{\s}_{i,i}\neq 0),       \\
	  & \lb{e_i,\lb{e_{i+ 1},\lb{e_i,e_{i- 1}}}}=\lb{f_i,\lb{f_{i+ 1},\lb{f_i,f_{i- 1}}}}=0                             & (mn\neq 2, a^{\s}_{i,i}=0), \\
	  & \lb{e_{i+1},\lb{e_{i-1},\lb{e_{i+1},\lb{e_{i-1},e_i}}}}=\lb{e_{i-1},\lb{e_{i+1},\lb{e_{i-1},\lb{e_{i+1},e_i}}}} & (mn=2, a^{\s}_{i,i}\neq 0), \\
	  & \lb{f_{i+1},\lb{f_{i-1},\lb{f_{i+1},\lb{f_{i-1},f_i}}}}=\lb{f_{i-1},\lb{f_{i+1},\lb{f_{i-1},\lb{f_{i+1},f_i}}}} & (mn=2, a^{\s}_{i,i}\neq 0), 
\end{align*}
where $\lb{X,Y}=[X,Y]_{q^{-\langle\beta|\gamma\rangle}}$ if $t_iXt_i^{-1}=q^{\langle\alpha_i|\beta\rangle}$ and $t_iYt_i^{-1}=q^{\langle\alpha_i|\gamma\rangle}$ for $\beta,\gamma\in\mc Q_\s$ and  $i\in I$.\qed
\end{dfn}

Note that the element $t_0t_1\cdots t_{\ka-1}$ is central and $\uqhsls$ for different $\s\in S_{m|n}$ are isomorphic. 

The superalgebra $\uqhsls$ is endowed with a coproduct $\Delta$ given by
\begin{equation}\label{eq:coproduct}
\Delta(e_i)=e_i\otimes t_i+1\otimes e_i,\quad \Delta(f_i)=f_i\otimes 1+t_i^{-1}\otimes f_i,\quad \Delta(t_i)=t_i\otimes t_i.
\end{equation}
The subalgebra of $\uqhsls$ generated by $e_i,f_i,t_i$, $i\in I$, is isomorphic to $\uqsls$ as a Hopf subalgebra.

\medskip 

The superalgebra $\uqhsls$ admits another presentation as follows.

Let $\delta(z)=\sum_{r\in \Z}z^{r}$ be the formal delta function. For $k\in \Z$, set $[k]=\frac{q^k-q^{-k}}{q-q^{-1}}$.
\begin{dfn}[\cite{Yam99} New Drinfeld Presentation]
The superalgebra  $\uqhsls$ is generated by the {\it current generators} $x^\pm_{i,r}, h_{i,r}$, $k^{\pm 1}_i, c^{\pm 1}$, $i \in I$, $r\in \Z'$. Here and below, we use the following convention: $r\in\Z'$ means $r\in \Z$ if $r$ is an index of a non-Cartan current generator $x^\pm_{i,r}$, and $r\in\Z'$ means $r\in \Z\setminus\{0\}$ if $r$ is an index of a Cartan current generator $h_{i,r}$. The parity of current generators is given by $|x^\pm_{i,r}|=|i|=(1-s_is_{i+1})/2$ while all remaining generators have parity $0$. The defining relations are as follows:
\begin{align*}
	&\text{$c$ is central},\quad k_ik_j=k_jk_i,\quad k_ik_i^{-1}=k_i^{-1}k_i=1,\quad k_ix^\pm_j(z)k_i^{-1}=q^{\pm a^{\s}_{i,j}}x^\pm_j(z),\\
	&[h_{i,r},h_{j,s}]=\delta_{r+s,0}\,\frac{[ra^{\s}_{i,j}]}{r}\frac{c^r-c^{-r}}{q-q^{-1}},\\
	&[h_{i,r},x^{\pm}_j(z)]=\pm\frac{[ra^{\s}_{i,j}]}{r}c^{-(r\pm|r|)/2}z^rx^\pm_j(z),\\
	&[x^+_i(z),x^-_j(w)]=\frac{\delta_{i,j}}{q-q^{-1}}\Bigl(\delta\Big(c\frac{w}{z}\Big)k_i^+(w)-\delta\Big(c\frac{z}{w}\Big)k_i^-(z)\Bigr),\\
	  & (z-q^{\pm a^{\s}_{i,j}}w)x^\pm_i(z)x^\pm_j(w)+(-1)^{|i||j|}(w-q^{\pm a^{\s}_{i,j}}z)x^\pm_j(w)x^\pm_i(z)=0 & (a^{\s}_{i,j}\neq 0),              \\
	  & [x^\pm_i(z),x^\pm_j(w)]=0                                                                                  & (a^{\s}_{i,j}=0),                  \\
	  & \mathrm{Sym}_{z_1,z_2}\lb{x^\pm_i(z_1),\lb{x^\pm_i(z_2),x^\pm_{i\pm 1}(w)}}=0\,                                    & (a^{\s}_{i,i}\neq 0,\ i\pm1\in I), \\
	  & \mathrm{Sym}_{{z_1,z_2}}\lb{x^\pm_i(z_1),\lb{x^\pm_{i+ 1}(w_1),\lb{x^\pm_i(z_2),x^\pm_{i- 1}(w_2)}}}=0             & (a^{\s}_{i,i}=0,\ i\pm 1 \in I),   
\end{align*}
where $x^\pm_i(z)=\sum_{r\in \Z}x^\pm_{i,r}z^{-r}$ and 
\[
k_i^\pm(z)=k_i^{\pm 1}\exp \Big(\pm (q-q^{-1})\sum_{r>0}h_{i,\pm r}z^{\mp r}\Big)=k_i^{\pm 1}+\sum_{r\gge 1}k_{i,\pm r}^\pm z^{\mp r}.
\]
Here and below, $\mathrm{Sym}_{{z_1,z_2}}$ stands for the symmetrization map on $z_1,z_2$. For instance,
\[
\mathrm{Sym}_{z_1,z_2}\lb{x^\pm_i(z_1),\lb{x^\pm_i(z_2),x^\pm_{i\pm 1}(w)}}=\lb{x^\pm_i(z_1),\lb{x^\pm_i(z_2),x^\pm_{i\pm 1}(w)}}+\lb{x^\pm_i(z_2),\lb{x^\pm_i(z_1),x^\pm_{i\pm 1}(w)}}.
\]
\end{dfn}

An isomorphism between Drinfeld-Jimbo and new Drinfeld presentations is given by
\begin{align*}
		  & e_i\mapsto x^+_{i,0},\quad f_i\mapsto x^-_{i,0},\quad t_i\mapsto k_i \qquad\qquad\qquad\qquad\qquad\qquad\qquad\qquad\qquad\qquad\qquad\qquad(i \in I), \\
		&t_0\mapsto c(k_1k_2\cdots k_{\ka-1})^{-1},\\
		&e_0\mapsto (-1)^ns_{\ka}[x^-_{m+n-1,0},\cdots,[x^-_{2,0},x^-_{1,1}]_{q_2^{-1}}\cdots]_{q_{\ka-1}^{-1}}(k_1k_2\cdots k_{\ka-1})^{-1},\\
		&f_0\mapsto s_{\ka}k_1k_2\cdots k_{\ka-1}[\cdots [x^+_{1,-1},x^+_{2,0}]_{q_2},\cdots, x^+_{\ka-1,0}]_{q_{\ka-1}},
\end{align*}
where $q_i=q^{s_i}$ for $i\in I$, see e.g. \cite[Theorem 5.2]{Zha14} and \cite{LYZ22}. Note that $t_0t_1\cdots t_{\ka-1}\mapsto c$.
	
\subsection{Representations of quantum affine superalgebras}
For simplicity, let $$\mathscr U:=\uqhsls/(c-1)$$ be the quantum loop superalgebra. Here and below we use the same notation for the images of the generators in $\uqhsls$ under the quotient. 

The quantum loop superalgebra $\mathscr U$ has a Hopf superalgebra structure inherited from $\uqhsls$. We shall need the following coproduct formula.% and set $$\mathscr X_\pm:=\sum_{i\in I}\sum_{r\in\Z}\C x_{i,r}^\pm\subset \mathscr U.$$

We start with introducing necessary notations.

There is a natural $\mc Q_\s$-grading on $\mathscr U$ given by
\[
(\mathscr U)_{\alpha}:=\{u\in \mathscr U ~|~ k_iuk_i^{-1}=q^{\langle\alpha_i|\alpha\rangle}u,\text{ for }i\in I\},\quad \alpha \in \mc Q_\s.
\]
Let $\mc Q_{\s}^{+}$ be the positive root lattice, $\mc Q_{\s}^{+}:=\bigoplus_{i\in I}\Z_{\gge 0}\alpha_i$. Define the length function $\imath:\mc Q_{\s}^{+}\to \Z_{\gge 0}$ by
\[
\imath\big(\sum_{i\in I}n_i\alpha_i\big)=\sum_{i\in I}n_i.
\]
Moreover, whenever $\imath(\alpha)$ is used, we implicitly assume that $\alpha \in \mc Q_{\s}^{+}$. Note that $\imath$ depends on $\s$. However, we shall not write $\s$ explicitly. Finally, for $i \in I$, let $\mathscr U_i^\pm$ be the subalgebra of $\mathscr U$ generated by $x_{j,0}^\pm$ for $j\in I\setminus\{i\}$.

\begin{prop}\label{prop:copro}
Let $i\in I$ and $r\in \Z$. We have the following properties for the coproduct of $\mathscr U$,
\begin{enumerate}
\item modulo $\sum_{\alpha\in \mc Q_\s^+\setminus\{0\}}(\mathscr U)_{\alpha}\otimes(\mathscr U)_{-\alpha}$,
\[\Delta(k_i^{\pm}(z))\equiv  k_i^{\pm}(z)\otimes k_i^{\pm}(z),\]
\item modulo $\sum_{\imath(\alpha)>1}(\mathscr U_i^+)_{\alpha}\otimes(\mathscr U)_{\alpha_i-\alpha}+\sum_{\imath(\alpha-\alpha_i)>0}(\mathscr U)_{\alpha}\otimes(\mathscr U)_{\alpha_i-\alpha}$,
\begin{align*}
&\Delta(x_{i,r}^+)\equiv x_{i,r}^+\otimes k_i+1\otimes x_{i,r}^+ +\sum_{j=1}^{r} x_{i,r-j}^+\otimes k_{i,j}^+, &(r\gge 0),\\
&\Delta(x_{i,r}^+)\equiv x_{i,r}^+\otimes k_i^{-1}+1\otimes x_{i,r}^+ +\sum_{j=1}^{-r-1} x_{i,r+j}^+\otimes k_{i,-j}^-, &(r< 0),
\end{align*}
\item modulo $\sum_{\imath(\alpha)>1}(\mathscr U)_{\alpha-\alpha_i}\otimes(\mathscr U_i^-)_{-\alpha}+\sum_{\imath(\alpha-\alpha_i)>0}(\mathscr U)_{\alpha-\alpha_i}\otimes(\mathscr U)_{-\alpha}$,
\begin{align*}
&\Delta(x_{i,r}^-)\equiv x_{i,r}^-\otimes 1+k_i\otimes x_{i,r}^- +\sum_{j=1}^{r-1} k_{i,j}^+\otimes x_{i,r-j}^-, &(r> 0),\\
&\Delta(x_{i,r}^-)\equiv x_{i,r}^-\otimes 1+k_i^{-1}\otimes x_{i,r}^- +\sum_{j=1}^{-r} k_{i,-j}^-\otimes x_{i,r+j}^-, &(r\leqslant 0).
\end{align*}
\end{enumerate}
\end{prop}
\begin{proof}
The proof follows from that of \cite[Proposition 5.4]{Zha14} and \cite[Proposition 3.6]{Zha16}, cf. also \cite[Proposition 4.4]{CP91}. We only sketch the key points. 

First, one shows as in \cite[Lemma 5.3]{Zha14} and \cite[Lemma A.3]{Zha16} that for any $s\in I$, 
\[
\Delta(h_{i,1})\equiv h_{i,1}\otimes 1+1\otimes h_{i,1}+\sum_{j\in I,j=i-1}^{i+1} \nu_{i,j} x_{j,0}^+\otimes x_{j,1}^{-},
\]
modulo $\sum_{\imath(\alpha)>1}(\mathscr U_s^+)_{\alpha}\otimes(\mathscr U)_{\alpha}+\sum_{\imath(\alpha-\alpha_s)>0}(\mathscr U)_{\alpha}\otimes (\mathscr U)_{-\alpha}$, where $\nu_{i,j}\in \C^\times$ and $\nu_{i,i+1}=q-q^{-1}$. Then by induction, one proves the formula for $\Delta(x_{i,r}^+)$, $r\gge 0$, as in \cite[Lemma A.4]{Zha16}. In this step, one needs to calculate explicitly certain coefficients which are obvious from the commutator relations, cf. \cite[Proposition 5.4]{Zha14}. Similarly, one obtains the other coproduct formulas for $\Delta(x_{i,r}^\pm)$ with proper modifications. The proof for $\Delta(k_i^\pm(z))$ is parallel to that of  \cite[Corollary A.5]{Zha16}.
\end{proof}

Given a parity sequence $\bm s\in S_{m|n}$, define a map $\mu_{\bm s}:\,\hat I\to \Z$ by
\[
\mu_{\bm s} (i):=\sum_{j=1}^i s_j,\qquad i\in\hat I
\]
where, by convention, $\mu_{\bm s}(0)=0$.

For $r\in\Z$, set
\[
\psi_r(z)=\frac{q^r-q^{-r}z}{1-z}.
\]
Then $\psi_{r}(0)\psi_{r}(\infty)=1$.

For a rational function $\phi(z)$ such that $\phi(0)\phi(\infty)=1$, denote by $\phi^\pm(z)$ the expansions of $\phi(z)$ as power series in $z^{\mp 1}$, respectively.
\begin{eg}[{cf. \cite[Lemma 3.1]{BM21}}]\label{eg:vector-rep}
Let $\mathscr V_{\bm s}\cong \C^{m|n}$ be the superspace with a basis $v_j$ for $1\lle j\lle \ka$ such that $|v_j|=(1-s_{j})/2$. Let $\xi$ be a formal variable and set $\mathscr V_{\bm s}(\xi):=\C[\xi^{\pm 1}]\otimes\mathscr V_{\bm s}$. Then $\mathscr V_{\bm s}(\xi)$ has a $\uqhsls$-module structure as follows,
\[
x_i^+(z)\xi^rv_j=\delta_{i+1,j}\delta(q^{\mu_{\bm s}(i)}\xi/z)\xi^rv_{j-1},
\]
\[
x_i^-(z)\xi^rv_j=s_{j}\delta_{i,j}\delta(q^{\mu_{\bm s}(i)}\xi/z)\xi^rv_{j+1},
\]
\[
k_{i}^\pm(z)\xi^r v_j=\begin{cases}
\psi_{s_i}^\pm(q^{\mu_{\bm s}(i)}\xi/z)\xi^rv_{j},\quad &\text{if }i=j,\\
\psi_{-s_{i+1}}^\pm(q^{\mu_{\bm s}(i)}\xi/z)\xi^rv_{j},\quad &\text{if }i=j-1,\\
\xi^rv_{j}, &\text{otherwise}.
\end{cases}
\]
Moreover, $c$ acts by identity.\qed
\end{eg}

One can also specialize $\xi$ to a nonzero complex number $a$. Then the same relations define a $\uqhsls$-module structure on $\mathscr V_{\bm s}$ which is the evaluation vector representation at the evaluation parameter $a$. We denote it by $\mathscr V_{\bm s}(a)$.

The action of Chevalley generators on $\Vs(a)$ is given by
\beq\label{eq:finite-action}
e_i(v_j)=\delta_{i+1,j}v_{j-1},\quad f_i(v_j)=s_j\delta_{i,j}v_{j+1},\quad t_i(v_j)=q^{s_j(\delta_{i,j}-\delta_{i+1,j})}v_j,\qquad (i\in I),
\eeq
\[
e_0(v_j)=\delta_{1,j}av_\ka,\qquad f_0(v_j)=s_{\ka}\delta_{\ka,j}a^{-1}v_1,\qquad t_0(v_j)=q^{s_j(\delta_{\ka,j}-\delta_{1,j})}v_j.
\]
The restriction of $\Vs(a)$ to a $\uqsls$-module is called the {\it vector} (or {\it natural}) {\it representation} of $\uqsls$.

A $\uqhsls$-module (resp. $\uqsls$-module) is called {\it integrable} if it is a weight module over $\uqsls$ and $e_i$, $f_i$, for all $i\in \hat I$ (resp. $i\in I$), act locally nilpotent. Clearly, $\Vs(\xi)$ and $\Vs(a)$ are integrable.

\medskip

Let $\ell\in\Z_{>0}$ and let $\bm\xi=(\xi_1,\dots,\xi_\ell)$ be a sequence of commuting formal variables. Denote by $\mathscr V_{\bm s}(\bm\xi)$ the tensor product $\mathscr V_{\bm s}(\xi_1)\otimes \mathscr V_{\bm s}(\xi_2)\otimes\cdots\otimes \mathscr V_{\bm s}(\xi_\ell)$. Then $\mathscr V_{\bm s}(\bm \xi)$ is a $\uqhsls$-module induced by the coproduct \eqref{eq:coproduct}. Note that the $\uqhsls$-action on $\Vs(\bm\xi)$ commutes with multiplication by the elements of  $\C[\xi_1^{\pm1},\dots,\xi_\ell^{\pm1}]$.

Let $\sfe_{\theta}$, $\sff_{\theta}$, $\sfk_{\theta}$ in $\End(\Vs)$ be defined by
\beq\label{eq:theta-oper1}
\sfe_{\theta}v_j=\delta_{j,\ka}v_1,\qquad \sff_\theta v_j=\delta_{j,1}v_\ka,\qquad \sfk_\theta v_j=q^{s_j(\delta_{j,1}-\delta_{j,\ka})}v_j. 
\eeq
Define 
\beq\label{eq:theta-oper2}
\sfe_{\theta,j}=\sfk_\theta^{\otimes j-1}\otimes \sfe_\theta\otimes 1^{\otimes \ell-j}, \qquad\sff_{\theta,j}=1^{\otimes j-1}\otimes \sff_\theta\otimes (\sfk_\theta^{-1})^{\otimes \ell-j}
\eeq
in $\End(\Vsl)$. Then for any $\bm v\in \Vsl\subset \Vs(\bm\xi)$, we have
\beq\label{eq:z0-oper}
e_0 \bm v =\sum_{j=1}^\ell \xi_j\sff_{\theta,j}\bm v,\qquad 
f_0 \bm v =s_{\ka}\sum_{j=1}^\ell \xi_j^{-1}\sfe_{\theta,j}\bm v,\qquad 
t_0 \bm v = (\sfk_\theta^{-1})^{\otimes \ell}\bm v, 
\eeq

For any sequence of integers $\bm j=(j_1,j_2,\dots,j_\ell)$, we say that $\bm j$ is {\it non-decreasing} if $1\lle j_1\lle \cdots\lle j_\ell\lle \kappa$. For any non-decreasing $\bm j$ and $1\lle r\lle \ell$, set 
$$
v_{\bm j}=v_{j_1}\otimes \cdots\otimes v_{j_\ell},\qquad \bm j_r^{\pm}=(j_1,\dots,j_{r-1},j_r\pm 1,j_{r+1},\dots,j_\ell),$$
and define
\beq\label{eq:parity-j}
|\bm j_r|=\sum_{a=1}^{r-1}|v_{j_a}|,\qquad \iota_\s(i,r;\bm j)=(-1)^{|i||\bm j_r|}.
\eeq

For nonnegative integers $a<b$, we use the convenient notation $(a,b]:=\{a+1,\dots,b\}$. Regard $\bm j$ as a map from $(0,\ell]$ to $(0,\ka]$.

For $a\in \C^\times$ (or any formal variable $\xi$) and a rational function $\phi(z)$ such that $\phi(0)\phi(\infty)=1$, define the normal order products by
\begin{align*}
:[\delta(a/z)\phi(z)]^+:=\phi^+(z)\sum_{r\gge 0}a^rz^{-r}+\phi^-(z)\sum_{r> 0}a^{-r}z^{r},\\
:[\delta(a/z)\phi(z)]^-:=\phi^+(z)\sum_{r> 0}a^rz^{-r}+\phi^-(z)\sum_{r\gge 0}a^{-r}z^{r}.
\end{align*}

\begin{prop}\label{prop:tensor-vect-rep}
If $\bm j$ is non-decreasing, we have the following in $\mathscr V_{\bm s}(\bm \xi)$,
\begin{align*}
&x_i^+(z)v_{\bm j}=\sum_{r=a_2+1}^{a_3}\iota_\s(i,r;\bm j)	:\Big[\delta(q^{\mu_{\bm s}(i)}\xi_r/z)\prod_{p=r+1}^{a_3}\psi_{-s_{i+1}}(q^{\mu_{\bm s}(i)}\xi_p/z)\Big]^+:v_{\bm j_{r}^-},\\
&x_i^-(z)v_{\bm j}=s_i\sum_{r=a_1+1}^{a_2}\iota_\s(i,r;\bm j)	:\Big[\delta(q^{\mu_{\bm s}(i)}\xi_r/z)\prod_{p=a_1+1}^{r-1}\psi_{s_{i}}(q^{\mu_{\bm s}(i)}\xi_p/z)\Big]^-:v_{\bm j_{r}^+},\\
&k_i^\pm(z)v_{\bm j}=\prod_{j_r=i}\psi^\pm_{s_{i}}(q^{\mu_{\bm s}(i)}\xi_r/z)\prod_{j_r=i+1}\psi^\pm_{-s_{i+1}}(q^{\mu_{\bm s}(i)}\xi_r/z)v_{\bm j},
\end{align*}
where $(a_1,a_2]=\bm j^{-1}(i)$ and $(a_2,a_3]=\bm j^{-1}(i+1)$.
\end{prop}
\begin{proof}
Since $\bm j$ is non-decreasing, the statement follows from Proposition \ref{prop:copro}. We only show it for the last equality. The first and second equalities are similar.

Consider the coproduct $\Delta(k_i^\pm(z))$ where the first factor acts on the first factor of $\mathscr V_{\bm s}(\bm \xi)$ while the second factors acts on the tensor product of the rest factors of $\mathscr V_{\bm s}(\bm \xi)$. Comparing the classical weights, it is clear that terms from $\sum_{\alpha\in \mc Q_\s^+\setminus\{0\}}(\mathscr U)_{\alpha}\otimes(\mathscr U)_{-\alpha}$ acting on $v_{\bm j}$ do not contribute to the final result of $k_i^\pm(z)v_{\bm j}$ as $\bm j$ is non-decreasing. Therefore, we have
\[
k_i^\pm(z)v_{\bm j}=k_i^\pm(z)v_{j_1}\otimes k_i^\pm(z)(v_{j_2}\otimes \cdots\otimes v_{j_\ell}).
\]As the subsequence $(j_2,\dots,j_\ell)$ is also non-decreasing, we can repeat the procedure and obtain the last equality by formulas in Example \ref{eg:vector-rep}.
\end{proof}

Note that, in general, the proposition does not hold if $\bm j$ is not non-decreasing.

\subsection{Quantum toroidal superalgebras}
We recall the definition of quantum toroidal superalgebra from \cite{BM19}.
Fix $d\in \mathbb{C}^\times$ and set
\beq\label{eq:para}
q_1=d\,q^{-1},\quad q_2=q^2,\quad q_3=d^{-1}q^{-1}.
\eeq

Note that $q_1q_2q_3=1$. We always assume that $q_1,q_2$ are generic, namely $q_1^{n_1}q_2^{n_2}q_3^{n_3}=1$, $n_1, n_2,n_3\in \mathbb{Z}$, if and only if  $n_1=n_2=n_3$. Also fix $d^{1/2}, q^{1/2}\in \mathbb{C}^\times$ such that $(d^{1/2})^2=d,\; (q^{1/2})^2=q$.

Define the matrix ${M}^{\s}=({m}^{\s}_{i,j})_{i,j\in \hat{I}}$ by $m^{\s}_{i+1,i}=-m^{\s}_{i,i+1}=s_{i+1}$, and $m^{\s}_{i,j}=0$, $i\neq j\pm 1$. Recall the affine Cartan matrix $\hat A^\s=(a_{i,j}^\s)_{i,j,\in \hat I}$ from \eqref{aff cartan}.

\begin{dfn}[\cite{BM19}]{\label{defE}}
	The \textit{quantum toroidal algebra associated with $\gl_{m|n}$} and parity sequence $\s$ is the unital associative superalgebra $\Es=\Es(q_1,q_2,q_3)$  generated by $E_{i,r},F_{i,r},H_{i,r}$, 
	and invertible elements $K_i$, $C$, where $i\in \hat{I}$, $r\in \Z'$, subject to the defining relations \eqref{relCK}-\eqref{Serre6} below.
	The parity of the generators is given by $|E_{i,r}|=|F_{i,r}|=|i|=(1-s_is_{i+1})/2$, and all remaining generators have parity $0$. We use generating series
\begin{align*}
	&E_i(z)=\sum_{r\in\Z}E_{i,r}z^{-r}, \qquad   \qquad \qquad\qquad           \qquad \qquad                               
	F_i(z)=\sum_{r\in\Z}F_{i,r}z^{-r}, \\                                         
	&K^{\pm}_i(z)=K_i^{\pm1}\exp\Bigl(\pm(q-q^{-1})\sum_{r>0}H_{i,\pm r}z^{\mp r}\Bigr)=K_i^{\pm 1}+\sum_{r\gge 1}K_{i,\pm r}^\pm z^{\mp r}. 
\end{align*}
Then the defining relations are as follows. 
\medskip

\noindent{\bf $C,K$ relations}
\begin{align}{\label{relCK}}
	&\text{$C$ is central}, 
	  &   & K_iK_j=K_jK_i,                             
	  &   & K_iE_j(z)K_i^{-1}=q^{a_{i,j}^{\s}}E_j(z),  
	  &   & K_iF_j(z)K_i^{-1}=q^{-a_{i,j}^{\s}}F_j(z). 
\end{align}
	
\noindent{\bf $K$-$K$, $K$-$E$ and $K$-$F$ relations}
\begin{align}
	  & K^\pm_i(z)K^\pm_j (w) = K^\pm_j(w)K^\pm_i (z),                                                                                              
	\label{KK1}\\
	  & \frac{d^{m_{i,j}^{\s}}C^{-1}z-q^{a_{i,j}^{\s}}w}{d^{m_{i,j}^{\s}}Cz-q^{a_{i,j}^{\s}}w}                                                      
	K^-_i(z)K^+_j (w) 
	=
	\frac{d^{m_{i,j}^{\s}}q^{a_{i,j}^{\s}}C^{-1}z-w}{d^{m_{i,j}^{\s}}q^{a_{i,j}^{\s}}Cz-w}
	K^+_j(w)K^-_i (z),
	\label{KK2}\\
	  & (d^{m_{i,j}^{\s}}z-q^{a_{i,j}^{\s}}w)K_i^\pm(C^{-(1\pm1)/2}z)E_j(w)=(d^{m_{i,j}^{\s}}q^{a_{i,j}^{\s}}z-w)E_j(w)K_i^\pm(C^{-(1\pm1) /2}z),   
	\label{KE}\\
	  & (d^{m_{i,j}^{\s}}z-q^{-a_{i,j}^{\s}}w)K_i^\pm(C^{-(1\mp1)/2}z)F_j(w)=(d^{m_{i,j}^{\s}}q^{-a_{i,j}^{\s}}z-w)F_j(w)K_i^\pm(C^{-(1\mp1) /2}z). 
	\label{KF}
\end{align}
\noindent{\bf $E$-$F$ relations}
\begin{align}\label{EF}
	  & [E_i(z),F_j(w)]=\frac{\delta_{i,j}}{q-q^{-1}} 
	(\delta\left(C\frac{w}{z}\right)K_i^+(w)
	-\delta\left(C\frac{z}{w}\right)K_i^-(z)).
\end{align}	
\noindent{\bf $E$-$E$ and $F$-$F$ relations}
\begin{align}
	  & [E_i(z),E_j(w)]=0\,, \quad [F_i(z),F_j(w)]=0\, \quad                                                                     & (a_{i,j}^{\s}=0),   \label{EE FF} 
	\\
	  & (d^{m_{i,j}^{\s}}z-q^{a_{i,j}^{\s}}w)E_i(z)E_j(w)=(-1)^{|i||j|}(d^{m_{i,j}^{\s}}q^{a_{i,j}^{\s}}z-w)E_j(w)E_i(z) \quad   & (a_{i,j}^{\s}\neq0),              
	\\
	  & (d^{m_{i,j}^{\s}}z-q^{-a_{i,j}^{\s}}w)F_i(z)F_j(w)=(-1)^{|i||j|}(d^{m_{i,j}^{\s}}q^{-a_{i,j}^{\s}}z-w)F_j(w)F_i(z) \quad & (a_{i,j}^{\s}\neq0).              
\end{align}
	
\noindent{\bf Serre relations}
\begin{align}
	  & \Sym_{{z_1,z_2}}\lb{E_i(z_1),\lb{E_i(z_2),E_{i\pm1}(w)}}=0\,\quad & (a^{\s}_{i,i}\neq 0),\label{Serre1} \\ 
	  & \Sym_{{z_1,z_2}}\lb{F_i(z_1),\lb{F_i(z_2),F_{i\pm1}(w)}}=0\,\quad & (a^{\s}_{i,i}\neq 0),               
	\label{Serre2}
\end{align}
	
If $mn\neq 2$,
\begin{align}
	  & \Sym_{{z_1,z_2}}\lb{E_i(z_1),\lb{E_{i+1}(w_1),\lb{E_i(z_2),E_{i-1}(w_2)}}}=0\,\quad & (a^{\s}_{i,i}= 0),\label{Serre3} 
	\\
	  & \Sym_{{z_1,z_2}}\lb{F_i(z_1),\lb{F_{i+1}(w_1),\lb{F_i(z_2),F_{i-1}(w_2)}}}=0\,\quad & (a^{\s}_{i,i}= 0).\label{Serre4} 
\end{align}
	
If $mn=2$,
\begin{align}
	\label{Serre5} & \Sym_{{z_1,z_2}}\Sym_{{w_1,w_2}}\lb{E_{i-1}(z_1),\lb{E_{i+1}(w_1),\lb{E_{{i-1}}(z_2),\lb{E_{i+1}(w_2),E_{i}(y)}}}}=    \\ \notag
	  & =\Sym_{{z_1,z_2}}\Sym_{{w_1,w_2}}\lb{E_{i+1}(w_1),\lb{E_{i-1}(z_1),\lb{E_{{i+1}}(w_2),\lb{E_{i-1}(z_2),E_{i}(y)}}}}\, &   & (a^{\s}_{i,i}\neq  0), \\
	\label{Serre6} & \Sym_{{z_1,z_2}}\Sym_{{w_1,w_2}}\lb{F_{i-1}(z_1),\lb{F_{i+1}(w_1),\lb{F_{{i-1}}(z_2),\lb{F_{i+1}(w_2),F_{i}(y)}}}}=    \\ \notag
	  & =\Sym_{{z_1,z_2}}\Sym_{{w_1,w_2}}\lb{F_{i+1}(w_1),\lb{F_{i-1}(z_1),\lb{F_{{i+1}}(w_2),\lb{F_{i-1}(z_2),F_{i}(y)}}}}\, &   & (a^{\s}_{i,i}\neq  0). \qedd
\end{align}
\end{dfn}
Note that the element $K_0K_1\cdots K_{\ka-1}$ is central and $\Es$ for different $\s\in S_{m|n}$ are isomorphic, see \cite{BM19}. 

We use the abbreviation $E_i:=E_{i,0}$ and $F_i:=F_{i,0}$ for $i\in \hat I$.

\subsection{Basics about quantum toroidal superalgebras}
Let $\bm s\in S_{m|n}$. Define the {\it vertical homomorphism} of superalgebras $\mathsf v_{\bm s}:\uqhsls\to \mathscr E_{\bm s}$ by
\[
\sfv_{\bm s}(x_i^+ (z))=E_i(d^{-\mu_{\bm s}(i)}z),\qquad \ \sfv_{\bm s}(x_i^- (z))=F_i(d^{-\mu_{\bm s}(i)}z),
 \]
\[
\sfv_{\bm s}(k_i^\pm (z))=K_i^{\pm}(d^{-\mu_{\bm s}(i)}z),\quad\ \ \  \sfv_{\bm s}(c)=C,\qquad\quad \    (i\in I).
\]
The map $\sfv_{\bm s}$ is injective for generic parameters. We call the image of $\sfv_{\bm s}$ the {\it vertical subalgebra} of $\mathscr E_{\s}$ and denote it by $\mathscr U_q^\sfv(\widehat{\fksl}_\s)$.

We  have an injective (for generic parameters) {\it horizontal homomorphism} of superalgebras $\sfh_{\bm s}:\uqhsls\to \mathscr E_\s$ given by
\[
e_i\mapsto E_{i},\quad f_i\mapsto F_{i},\quad t_i\mapsto K_i,\qquad (i\in \hat I).
\]
We call the image of $\sfh_{\s}$ the {\it horizontal subalgebra} of $\mathscr E_\s$ and denote it by $\mathscr U_q^\sfh(\widehat{\fksl}_\s)$.

Note that $\mathscr E_\s$ is generated by $\mathscr U_q^\sfv(\widehat{\fksl}_\s)$ and $\mathscr U_q^\sfh(\widehat{\fksl}_\s)$.

\medskip

For any $u\in \C^\times$, denote by $\gamma_{u,\s}:\mathscr E_\s\to \mathscr E_\s$ the {\it shift automorphism} by $u$ defined by
\[
\gamma_{u,\bm s}(C)=C,\qquad \gamma_{u,\bm s}(A_i(z))=A_i(uz),\qquad (i\in \hat I,\ A=K^\pm,E,F).
\]

Define a map $\tau:S_{m|n}\to S_{m|n}$ which sends $\s=(s_1,\dots,s_{\ka})$ to $\tau\s:=(s_\ka,s_1,\dots,s_{\ka-1})$. There exists an isomorphism of superalgebras $\widehat{\tau}_\s:\Es\to \mathscr E_{\tau \s}$ given by 
\beq\label{eq:tau}
\widehat{\tau}_\s(C)=C,\qquad \widehat{\tau}_\s(A_i(z))=A_{i+1}(q_1^{-s_\ka}z),\qquad (i\in \hat I,\ A=K^\pm,E,F).
\eeq

\medskip

An $\Es$-module $M$ has {\it trivial central charge} if the restrictions of $M$ to $\mathscr U_q^\sfv(\widehat{\fksl}_\s)$ and $\mathscr U_q^\sfh(\widehat{\fksl}_\s)$ also have trivial central charge. Namely, $C=1$ and $K_0K_1\cdots K_{\ka-1}=1$.

We say that a $\uqsls$-module is {\it of level $\ell$} if all its irreducible components are isomorphic some submodules of $\mathscr V_\s^{\otimes \ell}$. A $\uqhsls$-module or an $\Es$-module is said to be {\it of level $\ell$} if it is of level $\ell$ as a $\uqsls$-module.

Set $\mathcal P_{\ell}:=\{\bla=(\la_1,\dots,\la_\ka)\in\Z_{\gge 0}^{\kappa}~|~\la_1+\la_2+\dots+\la_\ka=\ell\}$. We call $\bla\in\mathcal P_{\ell}$ a {\it polynomial weight}.

An $\Es$-module 
 $M$ with trivial central charge and of level $\ell$ is {\it integrable} if $M$ is integrable as modules over $\mathscr U_q^\sfv(\widehat{\fksl}_\s)$ and $\mathscr U_q^\sfh(\widehat{\fksl}_\s)$, and
\[
M=\bigoplus_{\bla\in \mathcal P_\ell}M_{\bla},\qquad M_{\bla}=\{v\in M\ |\ K_iv=q^{s_i\la_i-s_{i+1}\la_{i+1}}v,\ i\in \hat I\}.
\]

\section{Super Schur-Weyl duality}
Since $\Es$ are all isomorphic for different $\s\in S_{m|n}$, in the rest of this paper, we shall set $\s$ to be the standard parity sequence or the images of the standard parity sequence under repeated application of $\tau$ for simplicity. However, our computations work for all parity sequences.

\subsection{Super Schur-Weyl duality for finite and affine cases}
We start with the super Schur-Weyl duality for finite case established in \cite{Moo03,Mit06}, cf. \cite{Jim86}.

Recall the vector representation $\Vs$ of $\uqsls$ and consider the linear map $\mathcal T:\Vs\otimes \Vs\to \Vs\otimes\Vs$ given by
\[
\mathcal T(v_i\otimes v_j)=\begin{cases}
s_iq^{1+s_i}v_i\otimes v_i, & \text{ if }i=j,\\
(-1)^{|v_i||v_j|}qv_j\otimes v_i, & \text{ if }i<j,\\
(-1)^{|v_i||v_j|}qv_j\otimes v_i+(q^2-1)v_i\otimes v_j, & \text{ if }i>j.	
\end{cases}
\]
Fix $\ell>1$. Let $\mathcal T_i\in \End(\mathscr V_\s^{\otimes \ell})$ be the map which acts on the $i$-th and $(i+1)$-st factors as $\mathcal T$, and the other factors as the identity.

Note that our choices of coproduct and $\mathcal T$ follow that of \cite{CP96,VV96} which are slightly different from that of \cite{Moo03,Mit06}.

\begin{thm}[\cite{Jim86,Moo03,Mit06}]\label{thm:finite}
There is a left $\mathbb H_\ell$-module structure on $\mathscr V_s^{\otimes \ell}$ such that $T_i$ acts as $\mathcal T_i$ for all $1\lle i<\ell$. Moreover, the action of $\mathbb H_{\ell}$ commutes with the action of $\uqsls$ on $\mathscr V_s^{\otimes \ell}$.

Let $M$ be a right $\ha$-module. Define $\mathcal J(M):=M\otimes_{\ha}\Vsl$ with the $\uqsls$-module structure induced by that on $\Vsl$. If $\ell< mn+\ka$, then the functor $\mc J:M\to \mc J(M)$ is an equivalence from the category of finite-dimensional $\ha$-modules to the category of finite-dimensional $\uqsls$-modules of level $\ell$. \qed
\end{thm}

The statement has been extended to the quantum affine superalgebra $\uqhsls$ in \cite{Fli18,KL20}, cf. \cite{GRV94,CP96}. 

We identify $\aha$ with $\aha^{(1)}$. Recall \eqref{eq:theta-oper1}, \eqref{eq:theta-oper2}, and the generators $Y_j$ in $\aha^{(1)}$.

\begin{thm}[\cite{GRV94,CP96,Fli18,KL20}]\label{thm:affine}
There exists a functor $\mc F$ from the category of finite-dimensional	right $\aha$-modules to the category of finite-dimensional $\uqhsls$-modules with trivial central charge and of level $\ell$, defined as follows. If $M$ is a right $\aha$-module, then $\mc F(M)=\mc J(M)$ as a $\uqsls$-module and the action of $e_0$, $f_0$, $t_0$ is given by
\beq\label{eq:v-0-action11}
e_0(w\otimes \bm v)=\sum_{j=1}^\ell wY_j^{-1}\otimes \sff_{\theta,j}\bm v,\quad f_0(w\otimes \bm v)=s_{\ka}\sum_{j=1}^\ell wY_j\otimes \sfe_{\theta,j}\bm v,\quad t_0(w\otimes \bm v)=w\otimes (\sfk_\theta^{-1})^{\otimes \ell}\bm v, 
\eeq
where $w\in M$ and $\bm v\in \Vsl$. Moreover, if $\ell<\ka$, then the functor $\mc F$ is an equivalence of categories.\qed
\end{thm}
Note that the $\uqhsls$-module $\mc F(M)$ can be understood as the tensor product of evaluation vector representations $\Vs(\bm Y)$, where $\bm Y=(Y_1^{-1},\dots,Y_\ell^{-1})$, with values in $M$, see Example \ref{eg:vector-rep} and Proposition \ref{prop:tensor-vect-rep}.

\subsection{Super Schur-Weyl duality for toroidal case}
Our main result is the Schur-Weyl duality between double affine Hecke algebra $\daha$ and the quantum toroidal superalgebra $\Es$, extending the main result of \cite{VV96} to the supersymmetric case.

Recall that $E_{i}$, $F_{i}$, $K_{i}$, $i\in \hat I$, are Chevalley generators of the horizontal subalgebra $\uhs$. It is also convenient to introduce extra generators $\sfE_0$, $\sfF_0$, $\sfK_0$ of $\Es$ so that combining with $E_{i}$, $F_{i}$, $K_{i}$, $i\in I$, they form Chevalley generators of the vertical subalgebra $\uvs$. Note that $E_{i}$, $F_{i}$, $K_{i}$, $i\in I$, are Chevalley generators of $\uqsls$.

Let $M$ be a right $\daha$-module. From Theorem \ref{thm:affine}, $\mc F(M)$ is a $\uvs$-module such that
\[
\sfE_0(w\otimes \bm v)=\sum_{j=1}^\ell wY_j^{-1}\otimes \sff_{\theta,j}\bm v,\quad\sfF_0(w\otimes \bm v)=s_{\ka}\sum_{j=1}^\ell wY_j\otimes \sfe_{\theta,j}\bm v,\quad\sfK_0(w\otimes \bm v)=w\otimes (\sfk_\theta^{-1})^{\otimes \ell}\bm v, 
\]
where the action of $A_{i}$,  for $i\in I$ and $A=E,F,K$, is as in Theorem \ref{thm:finite}. %Here the parameter $d$ appears because of the extra shifts in the vertical homomorphism $\sfv_\s$.
% \beq\label{eq:v-0-action2}
% \sfK_0(w\otimes \bm v)=w\otimes (\sfk_\theta^{-1})^{\otimes \ell}\bm v, 
% \eeq

\medskip

Recall $\zeta$ from Definition \ref{def:DAHA} and $q_1=dq^{-1}$ from \eqref{eq:para}. Our main result is the toroidal super Schur-Weyl duality. 

\begin{thm}\label{thm:toroidal}
If $\zeta=q_1^{n-m}$ and $\ka\gge 4$, then there exists a functor $\scrF$ from the category of right $\daha$-modules to the category of integrable $\Es$-modules with trivial central charge and of level $\ell$, defined as follows. If $M$ is a right $\daha$-module, then $\scrF(M)=\mc F(M)$ as a $\uvs$-module and the action of $E_{0}$, $F_{0}$, $K_{0}$ is given by
\beq\label{eq:zero-act1}
E_{0}(w\otimes \bm v)=\sum_{j=1}^\ell wX_j\otimes \sff_{\theta,j}\bm v,\quad F_{0}(w\otimes \bm v)=s_{\ka}\sum_{j=1}^\ell wX_j^{-1}\otimes \sfe_{\theta,j}\bm v,\quad K_{0}(w\otimes \bm v)=w\otimes (\sfk_\theta^{-1})^{\otimes \ell}\bm v, 
\eeq
where $w\in M$ and $\bm v\in \Vsl$. Moreover, if $\ell<\ka-2$, then the functor $\scrF$ is an equivalence of categories.
% \beq\label{eq:zero-act2}
% K_{0}(w\otimes \bm v)=w\otimes (\sfk_\theta^{-1})^{\otimes \ell}\bm v, 
% \eeq
\end{thm}

We shall prove the theorem in the next section. Before that, we make a few remarks which will be used later.

\begin{rem}\label{rem:integrable}
Since $q$ is not a root of unity, the $\ha$-modules and integrable $\uqsls$-modules are direct sums of finite-dimensional modules. (Note that in general the category of finite-dimensional $\uqsls$-modules	is not semisimple, however we restrict to the subcategory of polynomial modules only which is semisimple.) Therefore, if $\ell< mn+\ka$, Theorem \ref{thm:finite} implies indeed an equivalence between the category of $\ha$-modules and the category of integrable $\uqsls$-modules of level $\ell$.\qed
\end{rem}

\begin{rem}\label{rem:affine}
Similarly, if $q$ is generic and $\ell<\ka$, then Theorem \ref{thm:affine} gives an equivalence between the category of $\aha$-modules and the category of integrable $\uqhsls$-modules with trivial central charge and of level $\ell$.\qed
\end{rem}

\section{Proof of the main result}
In this section, we prove that the $\uvs$-action and the action of $E_{0},F_{0},K_{0}$ on $\mc F(M)$ extend to an $\Es$-module structure on $\mc F(M)$. Moreover, the resulting $\Es$-module $\scrF(M)$ is integrable with trivial central charge and of level $\ell$. Finally, we show that the functor $\scrF$ is an equivalence of categories if $\ell<\ka-2$.
\subsection{Explicit action of vertical subalgebra}
Clearly, any vector $w\otimes \bm v\in \scrF(M)$ can be written as $\sum_{\bm j}w_{\bm j}\otimes v_{\bm j}$ summed over non-decreasing $\bm j$, where $\bm j$ is an $\ell$-tuple of integers from $(0,\ka]$ and $w_{\bm j}\in M$. Hence it suffices for us to concentrate on $\bm v$ of the form $v_{\bm j}$ for non-decreasing $\bm j$.

We need the explicit action of Drinfeld currents of $\uvs$ on $\scrF(M)$ which follows directly from Proposition \ref{prop:tensor-vect-rep}.
\begin{cor}\label{cor:uvs-action}
If $\bm j$ is non-decreasing, we have the following in $\scrF(M)$,
\begin{align*}
&E_i(z)(w\otimes v_{\bm j})=\sum_{r=a_2+1}^{a_3}\iota_\s(i,r;\bm j)w:\Big[\delta(q_1^{\mu_{\bm s}(i)}Y_rz)\prod_{p=r+1}^{a_3}\psi_{s_{i+1}}(q_1^{\mu_{\bm s}(i)}Y_pz)\Big]^+:\otimes v_{\bm j_{r}^-},\\
&F_i(z)(w\otimes v_{\bm j})=s_i\sum_{r=a_1+1}^{a_2}\iota_\s(i,r;\bm j)	w:\Big[\delta(q_1^{\mu_{\bm s}(i)}Y_rz)\prod_{p=a_1+1}^{r-1}\psi_{-s_{i}}(q_1^{\mu_{\bm s}(i)}Y_pz)\Big]^-:\otimes v_{\bm j_{r}^+},\\
&K_i^\pm(z)(w\otimes v_{\bm j})=w\prod_{j_r=i}\psi^\pm_{-s_{i}}(q_1^{\mu_{\bm s}(i)}Y_rz)\prod_{j_r=i+1}\psi^\pm_{s_{i+1}}(q_1^{\mu_{\bm s}(i)}Y_rz)\otimes v_{\bm j},
\end{align*}
where $(a_1,a_2]=\bm j^{-1}(i)$, $(a_2,a_3]=\bm j^{-1}(i+1)$, $w\in M$ and $i\in I$. Moreover, $C$ acts by identity.
\end{cor}
\begin{proof}
Comparing \eqref{eq:z0-oper} with \eqref{eq:v-0-action11} and noting the shifts in the vertical homomorphism, the statement follows from Proposition \ref{prop:tensor-vect-rep}.	
\end{proof}

\subsection{An important proposition}
We shall define the action of the series $E_0(z),F_0(z),K_0^\pm(z)$ on $\mc F(M)$ in Section \ref{sec:proof-1}. To this end, we need the following linear map and its properties.

Let $\Psi_\s:M\otimes_{\ha}\Vsl\to M\otimes_{\ha}\mathscr V_{\tau \s}^{\otimes\ell}$ be the linear map defined by
\[
\Psi_\s(w\otimes v_{\bm j})=wX_1^{-\delta_{j_1,\ka}}X_2^{-\delta_{j_2,\ka}}\cdots X_\ell^{-\delta_{j_\ell,\ka}}\otimes v_{j_1+1}\otimes v_{j_2+1}\otimes\cdots \otimes v_{j_\ell+1},
\]
for any $\ell$-tuple $\bm j=(j_1,\dots,j_\ell)$, $w\in M$. Here by convention $v_{\ka+1}=v_1$. Note that the tensor product of vector representations in the target space is for superalgebras associated to the parity sequence $\tau\bm s$. In particular, we have $|v_j|$ in $\Vs$ coincides with $|v_{j+1}|$ in $\mathscr V_{\tau \s}$, and $\Psi_{\bm s}$ is an even linear map.

\begin{lem}
The linear map $\Psi_\s$ is well-defined.	
\end{lem}
\begin{proof}
It reduces to show that
\[
\Psi_\s(w\otimes \mathcal T_iv_{\bm j})=\Psi_\s(wT_i\otimes v_{\bm j})
\] for all $1\lle i<\ell$. It suffices to show it for the case of $\ell=2$.

We have four situations.
\begin{enumerate}
	\item If $j_1\ne\kappa$ and $j_2\ne \kappa$, this is obvious.
	\item If $j_1=\kappa$ and $j_2\ne \kappa$, one uses 
	$$
	T_1X_1^{-1}=(q^2-1)X_1^{-1}+q^2T_1^{-1}X_1^{-1}=(q^2-1)X_1^{-1}+X_2^{-1}T_1,
	$$
	which is obtained from $(T_1+1)(T_1-q^2)=0$ and $T_1X_1T_1=q^2 X_2$.
	\item If $j_1\ne \kappa$ and $j_2=\kappa$, this is clear from $T_1X_1T_1=q^2 X_2$.
	\item If $j_1=j_2=\kappa$, it follows from the fact that $T_1$ commutes with $X_1^{-1}X_2^{-1}$.\qedhere
\end{enumerate}
\end{proof}

We follow the main idea of \cite{VV96}. Recall that for $\s=(s_1,\dots,s_\ka)$, we have $\tau \s=(s_\ka,s_1,\dots,s_{\ka-1})$ and the superalgebra isomorphism $\widehat \tau_\s:\Es\to \mathscr E_{\tau \s}$, see \eqref{eq:tau}.

For $r\in\Z_{>0}$, define
\beq\label{eq:psi-power}
\Psi_\s^r:=\Psi_{\tau^{r-1}\s}\circ \Psi_{\tau^{r-2}\s}\circ\cdots\circ \Psi_{\tau\s}\circ \Psi_\s,\qquad \Psi_\s^{-r}:=(\Psi_\s^r)^{-1}.
\eeq
We also use the superscript $\s$ to distinguish generators from $\Es$ (also other notations) for various $\s$.

The following proposition is crucial in the proof of Theorem \ref{thm:toroidal}.
\begin{prop}\label{prop:main}
For $1<i<\ka$, we have the following identities in $\End(M\otimes_{\ha}\Vsl)$,
\begin{align*}
&\Psi_\s^{-1}\circ E_i^{\tau\s}(z)\circ \Psi_\s=E_{i-1}^\s(zq_1^{s_\ka}), && \Psi_\s^{-2}\circ E_1^{\tau^2\s}(\zeta z)\circ \Psi_\s^2=E_{\ka-1}^\s(zq_1^{n-m+s_{\ka-1}+s_{\ka}}),\\	
&\Psi_\s^{-1}\circ F_i^{\tau\s}(z)\circ \Psi_\s=F_{i-1}^\s(zq_1^{s_\ka}), && \Psi_\s^{-2}\circ F_1^{\tau^2\s}(\zeta z)\circ \Psi_\s^2=E_{\ka-1}^\s(zq_1^{n-m+s_{\ka-1}+s_{\ka}}),\\
&\Psi_\s^{-1}\circ K_i^{\pm,\tau\s}(z)\circ \Psi_\s=K_{i-1}^{\pm,\s}(zq_1^{s_\ka}), && \Psi_\s^{-2}\circ K_1^{\pm,\tau^2\s}(\zeta z)\circ \Psi_\s^2=K_{\ka-1}^{\pm,\s}(zq_1^{n-m+s_{\ka-1}+s_{\ka}}).
\end{align*}
\end{prop}
\begin{proof}
We only show identities in the first line. The rests are similar.

We start with the first one. If $1<i<\ka$, it suffices to show that the action of $E_i^{\tau\s}(z)\circ \Psi_\s$ and $\Psi_{\bm s}\circ E_{i-1}^\s(zq_1^{s_\ka})$ on $w\otimes v_{\bm j}^{\s}$ coincides for $w\in M$ and non-decreasing $\ell$-tuple $\bm j$. Put
\[
\bm j^{-1}(i-1)=(a_1,a_2],\quad \bm j^{-1}(i)=(a_2,a_3],\quad \bm j^{-1}(\ka)=(b,\ell],
\]
\[
\bm{j_1}=(j_1+1,\dots,j_{b}+1,1,\dots,1),\quad \bm{j_2}=(1,\dots,1,j_1+1,\dots,j_{b}+1).
\]
Then we have
$$
\bm{j_2}^{-1}(i)=(\ell-b+a_1,\ell-b+a_2],\qquad \bm{j_2}^{-1}(i+1)=(\ell-b+a_2,\ell-b+a_3].
$$
Recall $T_{j,i}$ from \eqref{eq:Tij} and $|\bm j_i|$ from \eqref{eq:parity-j}. Set
$$
R_{b}=(-1)^{(\ell-b)|{v^\s_\kappa||\bm{j}_{b+1}|}}q^{b(b-\ell)}T_{b,1}T_{b+1,2}\cdots T_{\ell-1,\ell-b}.
$$
Here and below, the notation of parity is always the one induced from $\s$.

On one hand, by Corollary \ref{cor:uvs-action}, we have
\begin{align*}
&\ E_i^{\tau\s}(z)\circ \Psi_\s(w\otimes v_{\bm j}^\s)\\=&\ E_i^{\tau\s}(z)(wX_{b}^{-1}X_{b+1}^{-1}\cdots X_{\ell}^{-1}\otimes v^{\tau\s}_{\bm{j_1}})
= E_i^{\tau\s}(z)(wX_{b}^{-1}X_{b+1}^{-1}\cdots X_{\ell}^{-1}R_b\otimes v^{\tau\s}_{\bm{j_2}})\\
=&\ \sum_{r=\ell-b+a_2+1}^{\ell-b+a_3}\iota_{\tau\s}(i,r;\bm{j_2}) wX_{b}^{-1}X_{b+1}^{-1}\cdots X_{\ell}^{-1}R_b:\Big[\delta(q_1^{\mu_{\tau\bm s}(i)}Y_rz)\prod_{p=r+1}^{\ell-b+a_3}\psi_{s_{i}}(q_1^{\mu_{\tau\bm s}(i)}Y_pz)\Big]^+:\otimes v^{\tau\s}_{(\bm{j_2})_{r}^-}.
\end{align*}
On the other hand, note that $\mu_{\tau\s}(i)=s_{\ka}+\mu_\s(i-1)$,  we have
\begin{align*}
&\ \Psi_{\s}\circ E_{i-1}^\s(zq_1^{s_\ka})(w\otimes v_{\bm j}^\s)\\
=&\ \Psi_\s\Big(\sum_{r=a_2+1}^{a_3}\iota_\s (i-1,r;\bm j)w:\Big[\delta(q_1^{\mu_{\tau\bm s}(i)}Y_rz)\prod_{p=r+1}^{a_3}\psi_{s_{i}}(q_1^{\mu_{\tau\bm s}(i)}Y_pz)\Big]^+:\otimes v^\s_{\bm j_{r}^-}\Big)\\
=&\ \sum_{r=a_2+1}^{a_3}\iota_\s (i-1,r;\bm j)w:\Big[\delta(q_1^{\mu_{\tau\bm s}(i)}Y_rz)\prod_{p=r+1}^{a_3}\psi_{s_{i}}(q_1^{\mu_{\tau\bm s}(i)}Y_pz)\Big]^+:X_{b}^{-1}X_{b+1}^{-1}\cdots X_{\ell}^{-1}\otimes v^{\tau\s}_{(\bm{j_1})_{r}^-}\\
=&\ \sum_{r=a_2+1}^{a_3}\iota_{\tau\s}(i,r+\ell-b;\bm{j_2})w:\Big[\delta(q_1^{\mu_{\tau\bm s}(i)}Y_rz)\prod_{p=r+1}^{a_3}\psi_{s_{i}}(q_1^{\mu_{\tau\bm s}(i)}Y_pz)\Big]^+:\\ &\qquad\qquad\qquad\qquad\qquad\qquad\qquad\qquad\qquad\qquad\qquad\qquad\times X_{b}^{-1}X_{b+1}^{-1}\cdots X_{\ell}^{-1}R_b\otimes v^{\tau\s}_{(\bm{j_2})_{r+\ell-b}^-},
\end{align*}
where in the last equality, we used
\[
\iota_{\tau\s}(i,r+\ell-b;\bm{j_2}) (-1)^{(\ell-b)|{v^\s_\kappa||\bm{j}_{b+1}|}}=\iota_\s (i-1,r;\bm j)(-1)^{(\ell-b)|{v^\s_\kappa|(|\bm{j}_{b+1}|-|v_i^{\s}|+|v_{i-1}^{\s}|)}}
\]
which follows from that the parity of $|i-1|$ is the same as that of $|v_{i-1}^{\s}|-|v_i^{\s}|$.

Recall $P_b=Q_{\ell-b,\ell-1}\cdots Q_{2,b+1}Q_{1,b}$ from \eqref{eq:Q-P}. It follows from Lemma \ref{lem:Q} that
\[
P_b:\Big[\delta(q_1^{\mu_{\tau\bm s}(i)}Y_rz)\prod_{p=r+1}^{a_3}\psi_{s_{i}}(q_1^{\mu_{\tau\bm s}(i)}Y_pz)\Big]^+:P_b^{-1}=:\Big[\delta(q_1^{\mu_{\tau\bm s}(i)}Y_{\ell-b+r}z)\prod_{p=r+1}^{a_3}\psi_{s_{i}}(q_1^{\mu_{\tau\bm s}(i)}Y_{\ell-b+p}z)\Big]^+:.
\]
Since $$R_b^{-1}X_{\ell}X_{\ell-1}\cdots X_{b+1}=(-1)^{(\ell-b)|{v^\s_\kappa||\bm{j}_{b+1}|}}q^{b(b-\ell)}P_b,$$
we conclude from the above equations that $E_i^{\tau\s}(z)\circ \Psi_\s(w\otimes v_{\bm j}^\s)=\Psi_{\s}\circ E_{i-1}^\s(zq_1^{s_\ka})(w\otimes v_{\bm j}^\s)$ and hence $$\Psi_{\s}^{-1}\circ E_i^{\tau\s}(z)\circ \Psi_\s= E_{i-1}^\s(zq_1^{s_\ka})$$ in $\End(M\otimes_{\ha}\Vsl)$ for $1<i<\ka$.

\medskip

We then consider the second one. Set 
$\bm j^{-1}(\ka-1)=(a,b]$, $\bm j^{-1}(\ka)=(b,\ell]$, and
\[
\bm{j_1}=(1,\dots,1,2,\dots,2,j_1+2,\dots,j_a+2),\quad
\bm{j_2}=(j_1+2,\dots,j_a+2,1,\dots,1,2,\dots,2)
\]
where
\[\bm{j_1}^{-1}(1)=(0,b-a],\quad
\bm{j_1}^{-1}(2)=(b-a,\ell-a], \quad
\bm{j_2}^{-1}(1)=(a,b],\quad 
\bm{j_2}^{-1}(2)=(b,\ell].
\]
Define $R_a=(-1)^{((b-a)|v^\s_{\ka-1}|+(\ell-b)|v^\s_\ka|)|\bm j_{a+1}|}q^{a(a-\ell)}T_{a,1}T_{a+1,2}\cdots T_{\ell-1,\ell-a}$. We have
\begin{align*}
&\	E_1^{\tau^2\s}(\zeta z)\circ \Psi_\s^2(w\otimes v_{\bm j}^\s)\\
=& E_1^{\tau^2\s}(\zeta z)(wX_{a+1}^{-1}\cdots X_\ell^{-1}\otimes v_{\bm{j_2}}^{\tau^2\s})=E_1^{\tau^2\s}(\zeta z)(wX_{a+1}^{-1}\cdots X_\ell^{-1}R_a\otimes v_{\bm{j_1}}^{\tau^2\s})\\
=&\ \sum_{r=b-a+1}^{\ell-a}\iota_{\tau^2\s}(1,r;\bm{j_1})wX_{a+1}^{-1}\cdots X_\ell^{-1}R_a:\Big[\delta(q_1^{s_{\ka-1}}\zeta Y_rz)\prod_{p=r+1}^{\ell-a}\psi_{s_{\ka}}(q_1^{s_{\ka-1}}\zeta Y_pz)\Big]^+:\otimes v^{\tau^2\s}_{(\bm{j_1})_{r}^-}.
\end{align*}
Note that $\mu_{\s}(\ka-1)+n-m+s_{\ka-1}+s_{\ka}=s_{\ka-1}$, we also have
\begin{align*}
&\ \Psi_\s^2\circ E_{\ka-1}^\s(zq_1^{n-m+s_{\ka-1}+s_{\ka}})	(w\otimes v_{\bm j}^\s)\\
=&\ \Psi_\s^2\Big(\sum_{r=b+1}^\ell \iota_{\s}(\ka-1,r;\bm j)w:\Big[\delta(q_1^{s_{\ka-1}}Y_rz)\prod_{p=r+1}^{\ell}\psi_{s_{\ka}}(q_1^{s_{\ka-1}}Y_pz)\Big]^+:\otimes v^{\s}_{\bm{j}_{r}^-}\Big)\\
=&\ \sum_{r=b+1}^\ell \iota_{\s}(\ka-1,r;\bm j)w:\Big[\delta(q_1^{s_{\ka-1}}Y_rz)\prod_{p=r+1}^{\ell}\psi_{s_{\ka}}(q_1^{s_{\ka-1}}Y_pz)\Big]^+:X_{a+1}^{-1}\cdots X_\ell^{-1}\otimes v^{\tau^2\s}_{(\bm{j_2})_{r}^-}\\
=&\ \sum_{r=b+1}^\ell \iota_{\tau^2\s}(1,r-a;\bm{j_1})w:\Big[\delta(q_1^{s_{\ka-1}}Y_rz)\prod_{p=r+1}^{\ell}\psi_{s_{\ka}}(q_1^{s_{\ka-1}}Y_pz)\Big]^+:X_{a+1}^{-1}\cdots X_\ell^{-1}R_a\otimes v^{\tau^2\s}_{(\bm{j_1})_{r-a}^-},
\end{align*}
where in the last equality we used that the parity of $|\ka-1|$ is the same as that of $|v_{\ka-1}^{\s}|-|v_\ka^{\s}|$. The rest is similar to the previous case by using Lemma \ref{lem:P}.
\end{proof}

\subsection{Proof of part 1}\label{sec:proof-1}
Now we define the action of the series $E_0(z),F_0(z),K_0^\pm(z)$ on $\mc F(M)=M\otimes_{\ha}\Vsl$ by
\begin{align*}
&E_0^\s(z)=\Psi_\s^{-1}\circ E_1^{\tau\s}(q_1^{-s_{\ka}}z)\circ \Psi_\s,\\
&F_0^\s(z)=\Psi_\s^{-1}\circ F_1^{\tau\s}(q_1^{-s_{\ka}}z)\circ \Psi_\s,\\
&K_0^{\pm,\s}(z)=\Psi_\s^{-1}\circ K_1^{\pm,\tau\s}(q_1^{-s_{\ka}}z)\circ \Psi_\s.
\end{align*}
If $\zeta=q_1^{n-m}$, then it follows from Proposition \ref{prop:main} that we have
\begin{align*}
&E_i^\s(z)=\Psi_\s^{-1}\circ E_{i+1}^{\tau\s}(q_1^{-s_{\ka}}z)\circ \Psi_\s,\\
&F_i^\s(z)=\Psi_\s^{-1}\circ F_{i+1}^{\tau\s}(q_1^{-s_{\ka}}z)\circ \Psi_\s,\\
&K_i^{\pm,\s}(z)=\Psi_\s^{-1}\circ K_{i+1}^{\pm,\tau\s}(q_1^{-s_{\ka}}z)\circ \Psi_\s,
\end{align*}
for all $i\in \hat I$ and any desired $\bm s$. Here we read indices modulo $\ka$. Recall the isomorphism $\widehat{\tau}_\s$ defined in \eqref{eq:tau}, then we have $\Psi_\s \circ E_i^{\s}(z)\circ \Psi_\s^{-1}=\widehat \tau_\s(E_i^{\s}(z))$ for all $i\in \hat I$.

Under this action, it is straightforward that \eqref{eq:zero-act1} is true. Thus, if these operators do define an $\Es$-action on $\mc F(M)$, then the $\Es$-module structure defined in Theorem \ref{thm:toroidal} is also well-defined. Moreover, these two $\Es$-module structures coincide. In particular, it is straightforward to check that $K_{0}K_{1}\cdots K_{\ka-1}(w\otimes \bm v)=w\otimes \bm v$ for all $w\in M$ and $\bm v\in\Vsl$.

Similarly to \eqref{eq:psi-power}, we use the convention
$$
\widehat{\tau}^r_\s:=\widehat\tau_{\tau^{r-1}\s}\circ \cdots \circ \widehat\tau_{\tau\s} \circ \widehat\tau_{\s},
$$
where $r\in\Z_{>0}$. To simplify the notation, we drop the dependence of $\s$ in $\widehat \tau_\s$ and $\Psi_\s$ but keep $\s$ in generating series.

\begin{proof}[Proof of Theorem \ref{thm:toroidal}, part 1]
By Corollary \ref{cor:uvs-action}, the operators $E_{i}^\s(z)$, $F_i^\s(z)$, $K_i^{\pm,\s}(z)\in \End(M\otimes_\ha\Vsl)$ satisfy relations in Definition \ref{defE} of $\Es$ for $i\in I$. To verify all the other relations, it suffice to check the relations involving $\widehat\tau^r(E_i^{\tau^{-r}\s}(z))$, $\widehat\tau^r(F_i^{\tau^{-r}\s}(z))$, $\widehat\tau^r(K_i^{\pm,\tau^{-r}\s}(z))$ for $r=1,\dots,\ka-1$ and $i\in I$ which are also  the relations of $\mathscr E_{\tau^{-r}\s}$ for $i\in I$. By construction, these operators are equal to $\Psi^r \circ E_i^{\tau^{-r}\s}(z)\circ \Psi^{-r}$, $\Psi^r \circ F_i^{\tau^{-r}\s}(z)\circ \Psi^{-r}$, $\Psi^r \circ K_i^{\pm,\tau^{-r}\s}(z)\circ \Psi^{-r}$, respectively. Since by Corollary \ref{cor:uvs-action}, $E_i^{\tau^{-r}\s}(z)$, $F_i^{\tau^{-r}\s}(z)$, $K_i^{\pm,\tau^{-r}\s}(z)$ satisfy the relations of $\mathscr E_{\tau^{-r}\s}$ for $i\in I$, we are done.

The fact that the $\Es$-module $\scrF(M)$ has trivial central charge is clear from Theorem \ref{thm:affine} and Corollary \ref{cor:uvs-action}. The integrability of $\scrF(M)$ follows from that of $\Vsl$. Moreover, by Theorem \ref{thm:finite}, $\scrF(M)$ is of level $\ell$, see Remark \ref{rem:integrable}.
\end{proof}

\subsection{Proof of part 2}
Assume for the reminder of the proof that $\ell<\ka-2$. We show that $\scrF$ is an equivalence of categories, which means we must prove that
\begin{enumerate}
\item (Surjectivity) every integrable $\Es$-module $\mathcal M$ with trivial central charge and of level $\ell$	is isomorphic to  $\scrF(M)$ for some $\daha$-module $M$;
\item (Fully faithfulness) $\scrF$ is bijective on sets of morphisms.
\end{enumerate}

We need the following useful lemma.
\begin{lem}\label{lem:inj}
\emph{(1)} If $\bm v$ is a generator of $\Vsl$ as a module over $\uqsls$, then $w\otimes \bm v\in M\otimes_{\ha}\Vsl$ is zero if and only if $w=0$.	\\
\emph{(2)} If $j_1,\dots,j_\ell\in (0,\ka]$ are pairwise distinct, then $v_{\bm j}$ is a generator of $\Vsl$ over $\uqsls$. 
\end{lem}
\begin{proof}
The first statement follows directly from Theorem \ref{thm:finite} and Remark \ref{rem:integrable} while the second one is clear.	
\end{proof}

\begin{proof}[Proof of Theorem \ref{thm:toroidal} part 2]
	Let $\mc M$ be an integrable $\Es$-module with trivial central charge and of level $\ell$.	Then the restriction of $\mc M$ to $\uvs$ is integrable with trivial central charge and of level $\ell$. Since $\aha^{(1)}$ is isomorphic to $\aha$, it follows from Theorem \ref{thm:affine} and Remark \ref{rem:affine} that there exists an $\aha^{(1)}$-module $M^{(1)}$ such that $\mc M\cong M^{(1)}\otimes_\ha\Vsl$ as $\uvs$-modules. Similarly, there exists an $\aha^{(2)}$-module $M^{(2)}$ such that $\mc M\cong M^{(2)}\otimes_\ha\Vsl$ as $\uhs$-modules. Moreover, these two modules $M^{(1)}$ and $M^{(2)}$ are isomorphic as $\ha$-modules. Hence we denote them by $M$.
	
	The action of $\uvs$ on $M\otimes_\ha \Vsl$ is as in Corollary \ref{cor:uvs-action} while the action of $E_{0},F_{0},K^\pm_{0}$ is as in \eqref{eq:zero-act1}. Note that the action of $X_i,Y_j\in \daha$ on $M$ is given by the $\aha^{(2)}$-module and the $\aha^{(1)}$-module structure of $M$, respectively. We would like to show that these two actions extend to an $\daha$-module structure on $M$. By Proposition \ref{prop:presentation}, it suffices to show that for any $w\in M$, we have
\beq\label{eq:toshow}
wQY_{i-1}Q^{-1}=wY_{i}\ \ (1< i\lle \ell),\qquad wQY_\ell Q^{-1}=\zeta wY_1,
\eeq
where $Q=X_1T_{1,\ell-1}=X_1T_1\cdots T_{\ell-1}$.

\medskip

We first show $wQY_{i-1}Q^{-1}=wY_{i}$ for $1<i\lle \ell$. Fix $1<i\lle\ell$. Set 
$$
\bm v=v_1\otimes\cdots \otimes v_i\otimes v_{i+2}\otimes \cdots\otimes v_{\ell+1},\quad \bm{\tl v}=v_2\otimes\cdots \otimes v_i\otimes v_{i+2}\otimes \cdots\otimes v_{\ell+1}\otimes v_{\ka}.
$$
Then it is clear that
\[
E_{0}(w\otimes \bm v)=(-1)^{|v_{\ka}|(|v_2|+\cdots+|v_i|+|v_{i+2}|+\cdots+|v_{\ell+1}|)}q^{1-\ell}wQ\otimes \bm{\tl v}.
\]We have
\begin{align*}
	E_{0}K_i^+(z)(w\otimes \bm v)=&\ E_{0}\big(w\psi_{-s_i}^+(q_1^{\mu_\s(i)}Y_iz)\otimes \bm v\big)\\
	=&\ (-1)^{|v_{\ka}|(|v_2|+\cdots+|v_i|+|v_{i+2}|+\cdots+|v_{\ell+1}|)}q^{1-\ell}w\psi_{-s_i}^+(q_1^{\mu_\s(i)}Y_iz)Q\otimes \bm{\tl v}
\end{align*}
and
\begin{align*}
	K_i^+(z)E_{0}(w\otimes \bm v)=&\ (-1)^{|v_{\ka}|(|v_2|+\cdots+|v_i|+|v_{i+2}|+\cdots+|v_{\ell+1}|)}q^{1-\ell}K_i^+(z)wQ\otimes \bm{\tl v}\\
	=&\ (-1)^{|v_{\ka}|(|v_2|+\cdots+|v_i|+|v_{i+2}|+\cdots+|v_{\ell+1}|)}q^{1-\ell}wQ\psi_{-s_i}^+(q_1^{\mu_\s(i)}Y_{i-1}z)\otimes \bm{\tl v}.
\end{align*}
Note that $E_{0}K_i^+(z)=K_i^+(z)E_{0}$ and $\bm{\tl v}$ is a generator of $\Vsl$ over $\uqsls$. It follows from Lemma \ref{lem:inj} that $w\psi_{-s_i}^+(q_1^{\mu_\s(i)}Y_iz)Q=wQ\psi_{-s_i}^+(q_1^{\mu_\s(i)}Y_{i-1}z)$. In particular, $wQY_{i-1}Q^{-1}=wY_{i}$ for $1<i\lle\ell$.

\medskip 

Then we show $wQY_\ell Q^{-1}=\zeta wY_1$. By taking the coefficients of $zw$ in \eqref{KE}, we have
\[
d^{-m_{i,j}^\s}(E_{j}K_{i,-1}-q^{a_{ij}^\s}K_{i,-1}E_{j})K_i=(q^{a_{ij}^\s}-q^{-a_{ij}^\s})E_{j,-1}.
\]Note that $m_{1,0}^\s=-a_{1,0}^\s=s_1$ and $m_{\ka-1,0}^\s=a_{\ka-1,0}^\s=-s_{\ka}$, we have
\[
s_1d^{-s_1}(E_{0}K_{1,-1}-q^{-s_1}K_{1,-1}E_{0})K_1=s_\ka d^{s_\ka}(E_{0}K_{\ka-1,-1}-q^{-s_\ka}K_{\ka-1,-1}E_{0})K_{\ka-1}.
\]
Set $\bm v=v_1\otimes v_3\otimes v_4\otimes \cdots\otimes v_{\ell+1}$ and $\bm{\tl v}=v_3\otimes v_4\otimes \cdots\otimes v_{\ell+1}\otimes v_{\ka}$. We have
\[
E_{0}(w\otimes\bm v)=(-1)^{|v_\ka|(|v_3|+\cdots+|v_{\ell+1}|)}q^{1-\ell}wQ\otimes \bm{\tl v}.
\]
A direct computation implies that
\begin{align*}
&\	s_1d^{-s_1}(E_{0}K_{1,-1}-q^{-s_1}K_{1,-1}E_{0})K_1(w\otimes \bm v)\\
=&\ s_1d^{-s_1}E_{0}K_{1,-1}K_1(w\otimes \bm v)
=(-1)^{|v_\ka|(|v_3|+\cdots+|v_{\ell+1}|)}q^{1-\ell}s_1(q^{-s_1}-q^{s_1})wY_1Q\otimes \bm{\tl v}.
\end{align*}
Similarly, one has
\begin{align*}
&\ s_\ka d^{s_\ka}(E_{0}K_{\ka-1,-1}-q^{-s_\ka}K_{\ka-1,-1}E_{0})K_{\ka-1}(w\otimes \bm v)\\
=&\	-s_\ka q_1^{s_\ka}K_{\ka-1,-1}E_{0}K_{\ka-1}(w\otimes \bm v)
= (-1)^{|v_\ka|(|v_3|+\cdots+|v_{\ell+1}|)}q^{1-\ell}s_\ka(q^{-s_\ka}-q^{s_\ka})q_1^{m-n}wQY_\ell\otimes \bm{\tl v},
\end{align*}
where we used $\ell+1 <\ka-1$ and $s_1+s_2+\cdots+s_\ka=m-n$. Since $\bm{\tl v}$ is a generator of $\Vsl$ over $\uqsls$, it follows from Lemma \ref{lem:inj} that $wY_1Q=q_1^{m-n}wQY_\ell$. Note that $\zeta=q_1^{n-m}$, we conclude that $wQY_\ell Q^{-1}=\zeta wY_1$.

\medskip
It remains to show that the functor $\scrF$ is fully faithful. The fact that $\scrF$ is injective on morphisms is clear from Theorem \ref{thm:affine}. To show $\scrF$ is surjective on morphisms, one only needs to use Theorem \ref{thm:affine} and the fact that $\daha$ is generated by the subalgebras $\aha^{(1)}$ and $\aha^{(2)}$.
\end{proof}

\end{document}